\documentclass[11pt,leqno]{amsart}
\usepackage{amsmath,amssymb,amsthm}
\usepackage{hyperref}
\newtheorem{theorem}{Theorem}[section]
\newtheorem{lemma}[theorem]{Lemma}

\newtheorem{claim}[theorem]{Claim}

\newtheorem{corollary}[theorem]{Corollary}
\theoremstyle{definition}
\newtheorem{definition}[theorem]{Definition}

\theoremstyle{remark}

\numberwithin{equation}{section}

\def\fnote#1{\footnote}

\def\ignora#1{}
\def\n3#1{\left\vert  \! \left\vert \! \left\vert \, #1 \, \right\vert \!
  \right\vert \! \right\vert }

\begin{document}

\title{ Lipschitz slices versus linear slices in Banach spaces }

\author{Julio Becerra Guerrero, Gin{\'e}s L{\'o}pez-P{\'e}rez and Abraham Rueda Zoca}
\address{Universidad de Granada, Facultad de Ciencias.
Departamento de An\'{a}lisis Matem\'{a}tico, 18071-Granada
(Spain)} \email{glopezp@ugr.es, juliobg@ugr.es,
arz0001@correo.ugr.es}
\thanks{The first author was partially supported by MEC (Spain) Grant MTM2014-58984-P and Junta de Andaluc\'{\i}a grants
FQM-0199, FQM-1215. The second author was partially supported by
MINECO (Spain) Grant MTM2015-65020-P and Junta de Andaluc\'{\i}a Grant
FQM-185.} 

\maketitle\markboth{J. Becerra, G. L\'{o}pez and A.
Rueda}{Diametral diameter two properties.}

\begin{abstract}
The aim of this note is study the topology generated by Lipschitz slices in the unit sphere of a Banach space. We prove that the above topology agrees with the weak topology in the unit sphere and, as a consequence, we obtain 
 Lipschitz characterizations of classical linear topics in Banach spaces, as Radon-Nikodym property, convex point of continuity property and strong regularity, which shows that the above classical linear properties only depend on the natural uniformity in the Banach space given by the metric and the linear structure. 
\end{abstract}

\section{Introduction}
 It
has been recently considered (see \cite{kmmw}) the study of the
Daugavet equation for Lipschitz operators, which has resulted in
the introduction of the so-called ``Lip-slices", getting a
charaterization of the Daugavet property in terms of such subsets
of the unit sphere. Motivated for these facts,
  we will define
the Lip-topology on the sphere of a Banach space as the topology
generated by the Lip-slices of the unit sphere and, even though it
seems natural to think that the weak topology is strictly coarser than
the Lip-topology, we will actually prove in Theorem \ref{caratopolo} that both topologies
agree, which will allow us to improve the
characterization of the Daugavet property given in  \cite[Theorem 2.6]{kmmw} and to extend the given one in \cite[Lemma 3]{shi} . Finally, we get as a consequence in Theorem
\ref{caralipsiso}  that a Banach space $X$ satisfies the Radon-Nikodym
property if, and only if, the unit sphere of every equivalent norm
in $X$ contains Lipschitz slices with arbitrarily small diameter,
which  shows that the well known characterization of
Radon-Nikodym property for Banach spaces in terms of slices is
uniform, that is, it only depends on the natural uniformity in $X$ given by the metric and the linear structure of the Banach space. Similar characterizations are too obtained for the convex point of continuity property and the strong regularity in Banach spaces. We refer to \cite{ggms} for background on these classical properties.

We shall introduce some notation. We consider real Banach spaces,
$B_X$ (resp. $S_X$) denotes the closed unit ball (resp. sphere) of
the Banach space $X$ and
$X^*$ stands for the dual space of $X$. A slice of a bounded
subset $C$ of $X$ is a set of the form
$$S(C,f,\alpha):=\{x\in C:\ f( x)>M-\alpha\},$$
where $f\in X^*$, $f\neq 0$, $M=\sup_{x\in C}f(x)$ and $\alpha>0$.   $w$   denotes the weak  topology of a Banach space and it is well known that the family of slices is a subbasis of the weak topology
for bounded subsets in $X$.

Let $X$ be a Banach space and $f:X\longrightarrow \mathbb R$ a
Lipschitz function. According to \cite{cla}, the generalized
derivative of at a point $x\in X$ in the direction $v\in X$ is
defined by
$$f^\circ(x,v):=\limsup\limits_{ y\rightarrow x,
t\searrow 0}\frac{f(y+tv)-f(y)}{t}.$$
Such a limit always exists from the Lipschitz condition.
Moreover, it is a sublinear and positively homogeneous function in
the variable $v$ \cite[Proposition 2.1.1]{cla}.

In addition, the generalized gradient of $f$ at $x$ is defined as
follows
$$\partial f(x):=\{x^*\in X^*: f^\circ(x,v)\geq x^*(v)\ \forall v\in X\}.$$
Given $v\in X$ it follows that \cite[Proposition 2.1.2]{cla}
$$f^\circ(x,v)=\max\limits_{x^*\in \partial f(x)}x^*(v)\ \ \forall x\in X.$$

\section{Lipschitz diameter two properties}

\begin{definition}

Let $X$ be a Banach space. According to \cite{kmmw}, a Lip-slice
of $S_X$ is a non-empty set of the form
$$S(S_X,f,\varepsilon):=\left\{\frac{x_1-x_2}{\Vert x_1-x_2\Vert}: \frac{f(x_1)-f(x_2)}{\Vert x_1-x_2\Vert}>\Vert f\Vert-\varepsilon \right\}.$$
\end{definition}

In \cite{kmmw} it is proved that if a Banach space has the
Daugavet property then, for each $x\in S_X$ , $S$ a Lip-slice of
$S_X$ and $\varepsilon>0$, there exists $y\in S$ such that
$$\Vert x+y\Vert>2-\varepsilon.$$
On the other hand, note that the class of Lip-slices defines a
subbasis of a topology on $S_X$. Indeed, define $\mathcal B$ to be
the class of Lip-slices of $S_X$. Then
$$\bigcup\limits_{S\in \mathcal B}S=S_X$$
as clearly $S(S_X,\Vert\cdot\Vert,\alpha)=S_X$ for each $\alpha\in
\mathbb R^+$ (indeed, given $x\in S_X$ one has that
$\frac{2x-x}{\Vert 2x-x\Vert}\in S$).

Now we can give the following

\begin{definition}

Let $X$ be a Banach space. We will define the Lip-topology on $S_X$ as the topology generated
by the subbasis $\mathcal B$.

\end{definition}

At that point it seems natural to wonder whether the Lip and the
weak topology on $S_X$ are equal. It is obvious, as slices of
$S_X$ are Lip-slices of $S_X$, that the weak topology is contained
in the Lip-topology. Our aim is to prove that both topologies actually
agree. The key of the proof will be the following

\begin{lemma}

Let $X$ be a Banach space and $S$ be a Lip-slice of $S_X$. Then for each $x\in S$ there exists $T$ a slice of $S_X$ such that
$$x\in T\subseteq S.$$
\end{lemma}

\begin{proof}
Assume that $S:=S(S_X,f,\varepsilon)$. Consider $x,y\in X, x\neq
y$ such that $\frac{y-x}{\Vert y-x\Vert}\in S$, i.e.
$$f(y)-f(x)>(1-\varepsilon)\Vert y-x\Vert.$$
Define $\phi:[0,1]\longrightarrow [x,y]$ by $\phi(t):=\lambda
y+(1-\lambda)x\ \ t\in [0,1]$ and
$F:=f\circ\phi:[0,1]\longrightarrow \mathbb R$. As $F$ is a
Lipschitz function we have that $F$ is derivable almost
everywhere. Now
$$(1-\varepsilon)\Vert y-x\Vert<f(y)-f(x)=F(1)-F(0)=\int_0^1 F'(t)\ dt.$$
From here we can find $t$ such that $F'(t)$ exists and it is
bigger that $(1-\varepsilon)\Vert y-x\Vert$. Thus given
$z=\phi(t)$ it follows
$$\limsup\limits_{h\rightarrow 0} \frac{f(z+h(y-x))-f(z)}{h}>(1-\varepsilon)\Vert y-x\Vert.$$
Indeed, given $h>0$ small enough one has
$$\phi(t+h)=(t+h)y+(1-(t+h))x=ty+(1-t)x+h(y-x).$$
As $ty+(1-t)x=\phi(t)=z$, we conclude that
$$\frac{F(t+h)-F(t)}{h}=\frac{f(z+h(y-x))-f(z)}{h}.$$
As there exists $\lim\limits_{h\rightarrow
0}\frac{F(t+h)-F(t)}{h}=\lim\limits_{h\rightarrow
0}\frac{f(z+h(y-x))-f(z)}{h}$ we conclude that
$$\lim\limits_{h\rightarrow 0}\frac{f(z+h(y-x))-f(z)}{h}=\limsup\limits_{h\rightarrow 0} \frac{f(z+h(y-x))-f(z)}{h}\leq f^\circ(z,y-x).$$
As the generalized derivative is positively homogeneous we
conclude that  $f^\circ(z,\frac{y-x}{\Vert
y-x\Vert})>1-\varepsilon$. Hence \cite[Proposition 2.1.2]{cla} there exists $\varphi\in \partial
f(x)$ such that $\varphi(\frac{y-x}{\Vert
y-x\Vert})>1-\varepsilon$. Consequently
$$\frac{y-x}{\Vert y-x\Vert}\in T:=\{v\in S_X\ /\ \varphi(v)>1-\varepsilon\}.$$
Now we shall prove that $T\subseteq S$. To this aim pick $v\in T$,
so $\varphi(v)>1-\varepsilon$. As
$$\varphi(v)\leq f^\circ(z,v)=\limsup\limits_{y\rightarrow z,t\searrow 0}\frac{f(y+tv)-f(y)}{t}$$
we can find $t>0$ and $y$ close enough to $z$ such that
$1-\varepsilon<\frac{f(y+tv)-f(y)}{t}$. From the definition of $S$
one has
$$S\ni \frac{y+tv-y}{\Vert y+tv-y\Vert}=\frac{tv}{\Vert tv\Vert}=v,$$
so $\frac{y-x}{\Vert y-x\Vert}\in T\subseteq S$.
\end{proof}

As a consequence we get the desired result.

\begin{theorem}\label{caratopolo}
Let $X$ be a Banach space. Then the weak and the Lip-topology agree on $S_X$.
\end{theorem}

\begin{proof}
As we have pointed out it is clear that the weak topology on $S_X$
is contained in the Lip-topology on $S_X$.

For the reverse inclusion, pick $W$ to be a basic open set of the
Lip-topology of $S_X$. Then $W:=\bigcap\limits_{i=1}^n
S(S_X,f_i,\varepsilon_i)$ is a finite and non-empty intersection
of Lip-slices. Pick $x\in W$. So, for each $i\in\{1,\ldots, n\}$,
we can find by above Lemma $T_i$ a slice of $S_X$ such that
$$x\in T_i\subseteq S(S_X,f_i,\varepsilon_i).$$
So, if we define $U:=\bigcap\limits_{i=1}^n T_i$, one has
$$x\in U\subseteq W,$$
so we are done.
\end{proof}

From here we can improve \cite[Corollary 2.6]{kmmw} and get a
characterization of Daugavet property which extends  
\cite[Lemma 3]{shi}.

\begin{corollary}

Let $X$ be a Banach space. Then $X$ has the Daugavet property if, and only if, for every $x\in
S_X$, nonempty open subset for the  Lip topology  $U$ of $S_X$  
and $\varepsilon>0$ there exists $y\in U$ such that
$$\Vert x+y\Vert>2-\varepsilon.$$
\end{corollary}

\begin{proof}
It is a straightforward application of \cite[Lemma 3]{shi} and the
fact that every finite and non-empty intersection of Lip-slices
contains a non-empty relatively weakly open subset of $S_X$.
\end{proof}

According to \cite[Definition 4.10]{ksw}, a locally convex
topology $\tau$ on a Banach space $X$ is said to be a
\textit{Daugavet topology} if given $x,y\in S_X$, $U$ a $\tau$
neighborhood of $y$ and $\varepsilon\in\mathbb R^+$ there exists
an element $z\in U\cap S_X$ such that $\Vert
x+z\Vert>2-\varepsilon$.

Above Corollary says that Lip-topology is a Daugavet topology in a
more general setting.

Now we shall exibit a characterization of Radon-Nikodym property (RNP), convex point of continuity property (CPCP) and strong regularity (SR) in terms of Lip-slices. (We refer to \cite{ggms} for background around these properties.) To this aim we shall need to prove a characterization of previous properties in terms of the diameter of slices (respectively non-empty relatively weakly open subsets, convex combination of slices) of the unit ball in each equivalent renorming of the space. It is well known that such characterization holds for the RNP (see \cite{DP}). Now we shall prove similar statements for CPCP and SR. To this aim we shall need the following two preliminary results.

\begin{lemma}\label{simetrizante}

Let $X$ be a Banach space. Then:

\begin{enumerate}
\item If $X$ fails to be strongly regular then, for each $\varepsilon>0$, there exists $C\subseteq X$ a closed, bounded, symmetric and convex subset such that $diam(K)=2$ and each convex combination of slices of $K$ has diameter, at least, $1-\varepsilon$.

\item If $X$ fails the CPCP then, for each $\varepsilon>0$, there exists $K\subseteq X$ a closed, convex and bounded such that $diam(K)=2$ and every non-empty relatively weakly open subset of $K$ has diameter, at least, $1-\varepsilon$.
\end{enumerate}
\end{lemma}

\begin{proof}
(1) Pick $\varepsilon>0$. By \cite[Proposition 4.10 (b)]{scsewe}, consider $C\subseteq X$ a closed and convex subset such that $diam(C)=1$ and such that each convex combination of slices has diameter, at least, $1-\varepsilon$. Now  $K:=\overline{co}(C\cup -C)$ is the desired subset by the proof of \cite[Lemma 2.12]{blr}.

(2) Pick $\varepsilon>0$. By \cite[Proposition 4.10, (a)]{scsewe} consider $C\subseteq X$ a closed, convex subset such that $diam(C)=1$ and such that every non-empty relatively weakly open subset of $C$ has diameter, at least, $1-\varepsilon$. Define $K:=\overline{C-C}$. $K$ is obviously closed. Moreover, $K$ is convex because $C-C$ is convex. Moreover, as $C-C$ is symmetric, so is $K$.

In order to finish the proof pick $U$ a weakly open subset of $X$ verifying $U\cap (C-C)\neq \emptyset$, and let us prove that last set has diameter, at least, $1-\varepsilon$. Pick $x_0\in U\cap (C-C)$, so $x_0=a-b$ for suitable $a,b\in C$. Now
$$x_0\in U\Rightarrow a=x_0+b\in b+U.$$

Consequently $a\in (b+U)\cap C$, so $(b+U)\cap C\neq \emptyset$. By the assumptions on $C$ we conclude that 
$$diam((b+U)\cap C)\geq 1-\varepsilon.$$
We will show that 
\begin{equation}\label{inclusime}
\left((b+U)\cap C\right)-b\subseteq U\cap (C-C).
\end{equation}
To this aim pick $x\in ((b+U)\cap C)-b$. So there exists $u\in U$ such that $x=b+u-b$. So $x=u\in U$. Moreover, $b+u\in C$ implies that $x=b+u-b\in C-C$. So (\ref{inclusime}) holds. Hence
$$1-\varepsilon\leq diam((b+U)\cap C)=diam(((b+U)\cap C)-b)\leq diam(U\cap (C-C)).$$ 
Last inequality proves that each non-empty relatively weakly open subset of $C-C$ has diameter, at least, $1-\varepsilon$. As each non-empty relatively weakly open subset of $K$ necessarily intersects $C-C$, we conclude that each non-empty relatively weakly open subset of $K$ has diameter, at least, $1-\varepsilon$.

It is also clear that $diam(K)=diam(C-C)=2$, so we are done.\end{proof}

Now we shall prove a renorming fact for Banach spaces failing CPCP or SR.

\begin{theorem}\label{teorenorma}
Let $(X,\Vert\cdot\Vert)$ be a Banach space  and consider $\alpha\in\mathbb R^+$.

\begin{enumerate}

\item  Assume that $C\subseteq  B_X$ is a closed, symmetric and convex subset verifiying that for every    non-empty relatively  weakly open subset $U$ of $C$ one has
$$diam_{\Vert\cdot\Vert}(U)\geq \alpha.$$
Then for each $\varepsilon\in\mathbb R^+$ there exists $ \vert\cdot\vert$ an equivalent norm on $X$ satisfying
\begin{enumerate}
\item $\vert x\vert\geq \Vert x\Vert\ \ \forall x\in X$
\item The diameter of every non-empty relatively weakly open subset of $ B_{\vert\cdot\vert}$ is, at least, $ \frac{\alpha}{1+\varepsilon}$.
\end{enumerate}

\item  Assume that $C\subseteq  B_X$ is a closed, symmetric and convex subset verifiying that  $\forall T\subseteq C$ convex combination of slices of $C$ it follows 
$$diam_{\Vert\cdot\Vert}(T)\geq \alpha.$$
Then for each $\varepsilon\in\mathbb R^+$ there exists $ \vert\cdot\vert$ an equivalent norm on $X$ satisfying
\begin{enumerate}
\item $\vert x\vert\geq \Vert x\Vert\ \ \forall x\in X$
\item The diameter of each convex combination of slices of $ B_{\vert\cdot\vert}$ is, at least, $ \frac{\alpha}{1+\varepsilon}$.
\end{enumerate}

\end{enumerate}

\end{theorem}

\begin{proof}

(1) Pick an arbitrary $\varepsilon\in\mathbb R^+, C\subseteq  B_X$ satisfying the hypothesis of the theorem.

Define on $X$ the norm whose unit ball is
$$\frac{1}{1+\varepsilon} \overline{C+\varepsilon B_X}.$$
This new norm on $X$ is equivalent to $\Vert\cdot \Vert$. Moreover, as $B_{\vert\cdot\vert}\subseteq B_X$, one has
$$\Vert x\Vert\leq \vert x\vert\ \ \forall x\in X.$$
Let us prove (b). To this aim pick $U\subseteq B_{\vert\cdot\vert}$ a non-empty relatively weakly open set. We can assume  that there exist $ x_0\in B_{\vert\cdot\vert}$, $\beta\in\mathbb R^+$ and $ x_1^*,\ldots,x_n^*\in X^*$ such that
$$U=(x_0+V_{\beta,x_1^*,\ldots,x_n^*})\bigcap \frac{1}{1+\varepsilon}\overline{C+\varepsilon  B_X},$$
where $V_{\beta,x_1^*,\ldots, x_n^*}:=\{x\in X: \vert x_i^*(x)\vert<\beta\ \forall i\in\{1,\ldots,n\}\}$. From linearity of the weak topology we can ensure the existence of $\gamma\in\mathbb R^+, y_1^*,\ldots, y_m^*\in X^*\ $ such that 
\begin{equation}\label{incluentocero}
 V_{\gamma,y_1^*,\ldots, y_m^*}+V_{\gamma,y_1^*,\ldots, y_m^*}\subseteq V_{\beta,x_1^*,\ldots, x_n^*}.
\end{equation}
Consider $ v\in (x_0+V_{\gamma,y_1^*,\ldots,y_m^*})\cap \left(\frac{1}{1+\varepsilon} (C+\varepsilon  B_X)\right)$. 
Then, on the one hand
\begin{equation}\label{v por bola}v=\frac{1}{1+\varepsilon} (c+\varepsilon b)\ \ c\in C, b\in B_X
\end{equation}
On the other hand
\begin{equation}\label{v por topologia lineal}v=x_0+x\ \ \ x\in V_{\gamma,y_1^*,\ldots,y_m^*}
\end{equation}
Now we can prove the following

\begin{claim}
The following inclusion holds
\begin{equation}\label{inclusion de entornos}
A:=\frac{1}{1+\varepsilon} \left (\left(c+V_{\frac{\gamma(1+\varepsilon)}{2},y_1^*,\ldots, y_m^*}\bigcap C\right)+\varepsilon \left (b+V_{\frac{\gamma (1+\varepsilon)}{2  \varepsilon},y_1^*,\ldots, y_m^*}\bigcap  B_X \right ) \right )
\end{equation}
$$ \subseteq (v+V_{\gamma,y_1^*,\ldots, y_m^*})\bigcap \left(\frac{1}{1+\varepsilon} (C+\varepsilon  B_X)\right).$$
\end{claim}

\begin{proof}
Pick $z\in A$. Then 
$$z=\frac{1}{1+\varepsilon}(c+x+\varepsilon(b+y)),$$
where $x\in V_{\frac{\gamma (1+\varepsilon)}{2},y_1^*,\ldots, y_m^*}$, $y\in V_{\frac{\gamma (1+\varepsilon)}{2 \varepsilon},y_1^*,\ldots, y_m^*}$, $c+x\in C$ and $ b+y\in B_X$.

As $c+x\in C$ and $ b+y\in B_X$ then 
$$z=\frac{1}{1+\varepsilon}(c+x+\varepsilon(b+y)\in \frac{1}{1+\varepsilon}(C+\varepsilon B_X).$$
On the other hand, bearing in mind that $v=\frac{1}{1+\varepsilon}(c+\varepsilon b)$, we conclude that 
$$z=v+\frac{x+\varepsilon y}{1+\varepsilon}.$$
So, in order to finish the proof of the claim, it is enough to prove that $\frac{x+\varepsilon y}{1+\varepsilon}\in V_{\gamma,y_1^*,\ldots, y_m^*}$. To this aim pick $i\in\{1,\ldots, m\}$. Then
$$\left\vert y_i^*\left( \frac{x+\varepsilon y}{1+\varepsilon}\right) \right\vert\leq \frac{\vert y_i^*(x)\vert+\varepsilon \vert y_i^*(y)\vert}{1+\varepsilon}<\frac{\frac{\gamma(1+\varepsilon)}{2}+\varepsilon \frac{\gamma (1+\varepsilon)}{2\varepsilon}}{1+\varepsilon}=\gamma,$$
so we are done.
\end{proof}

 Furthermore (\ref{v por topologia lineal}) implies
$$
(v+V_{\gamma,y_1^*,\ldots, y_m^*})\bigcap \left(\frac{1}{1+\varepsilon} (C+\varepsilon  B_X)\right)=$$ $$(x_0+x+V_{\gamma,y_1^*,\ldots, y_m^*})\bigcap \left(\frac{1}{1+\varepsilon} (C+\varepsilon  B_X)\right)
$$
$$\subseteq (x_0+V_{\gamma,y_1^*,\ldots, y_m^*}+V_{\gamma,y_1^*,\ldots, y_m^*})\bigcap \left(\frac{1}{1+\varepsilon} (C+\varepsilon  B_X)\right)$$ $$\subseteq (x_0+V_{\beta,x_1^*,\ldots, x_n^*})\bigcap \left(\frac{1}{1+\varepsilon} (C+\varepsilon  B_X)\right)\subseteq U.$$
 By assumption, we have that 
$$diam \left (\frac{1}{1+\varepsilon} \left (\left(c+V_{\frac{\gamma(1+\varepsilon)}{2},y_1^*,\ldots, y_m^*}\bigcap C\right)+\varepsilon \left (b+V_{\frac{\gamma (1+\varepsilon)}{2  \varepsilon},y_1^*,\ldots, y_m^*}\bigcap  B_X \right ) \right ) \right )$$ $$\geq \frac{\alpha}{1+\varepsilon},$$
hence
$$diam_{\vert\cdot\vert}(U)\geq diam_{\Vert\cdot\Vert}(U)\geq \frac{\alpha}{1+\varepsilon},$$
which proves (a).

The proof of (2) follows from Lemma \ref{simetrizante} and \cite[Lemma 2.13]{blr}.
\end{proof}

From above results and from \cite[Corollary 3.2]{scsewe} we get the following

\begin{theorem}\label{caraisoslices}
Let $X$ be a Banach space.

\begin{enumerate}

\item $X$ has the RNP if, and only if, each equivalent norm on $X$ verifies that its unit ball contains slices of arbitrarily small diameter.

\item $X$ has the CPCP if, and only if, each equivalent norm on $X$ verifies that its unit ball contains non-empty relatively weakly open subsets of arbitrarily small diameter.

\item $X$ is SR if, and only if, each  equivalent norm on $X$ verifies that its unit ball contains convex combinations of slices of arbitrarily small diameter.
\end{enumerate}

\end{theorem}

In order to characterize the CPCP and SR in terms of Lip-slices, we need the following lemma.

\begin{lemma}\label{renormadebil}
Let $X$ be a Banach space. The following assertions are equivalent:

\begin{enumerate}
\item $B_X$ contains non-empty relatively weakly open subsets of arbitrarily small diameter.
\item $S_X$ contains non-empty relatively weakly open subsets of arbitrarily small diameter.
\item $S_X$ contains non-empty relatively open subsets of the Lip topology of arbitrarily small diameter.
\end{enumerate}

\end{lemma}

\begin{proof}
(2)$\Leftrightarrow$(3) is known by Theorem \ref{caratopolo}.

(1)$\Rightarrow$(2) is trivial.

(2)$\Rightarrow$(1).

Consider $\delta>0$ and  $U:=\bigcap\limits_{i=1}^n S(S_X,x_i^*,\varepsilon_i)$ a basic non-empty relatively weakly open subset of $S_X$ whose diameter is smaller than $\frac{\delta}{2}$. Define
$$V:=\bigcap\limits_{i=1}^n S(B_X,x_i^*,\varepsilon_i),$$
which is a non-empty relatively weakly open subset of $B_X$ and pick $x\in V\cap S_X$. Pick $0<\alpha<\min\{\frac{\delta}{2},1\}$ and consider by \cite[Lemma 2.1]{ik}, for each $i\in\{1,\ldots,n\}$, a functional $y_i^*\in S_{X^*}$ such that
$$x\in\bigcap\limits_{i=1}^n S(B_X,y_i^*,\alpha)\subseteq V.$$
 Consequently $x\in \bigcap\limits_{i=1}^n S(S_X,y_i^*,\alpha)\subseteq U$, so $diam\left(\bigcap\limits_{i=1}^n S(S_X,y_i^*,\alpha)\right)<\frac{\delta}{2}$. Now we will prove
 $$diam\left(\bigcap\limits_{i=1}^n S(B_X,y_i^*,\alpha)\right)<\delta.$$
  To this aim pick $x,y\in \bigcap\limits_{i=1}^n S(B_X,y_i^*,\alpha)$. Now $\frac{x}{\Vert x\Vert}, \frac{y}{\Vert y\Vert}\in \bigcap\limits_{i=1}^n S(S_X,y_i^*,\alpha)$, so $\left\Vert \frac{x}{\Vert x\Vert}-\frac{y}{\Vert y\Vert} \right\Vert<\frac{\delta}{2}$. Hence
 $$\Vert x-y\Vert\leq \left\Vert \frac{x}{\Vert x\Vert}-\frac{y}{\Vert y\Vert} \right\Vert+\left\Vert x-\frac{x}{\Vert x\Vert}\right\Vert +\left\Vert y-\frac{y}{\Vert y\Vert}\right\Vert$$
 $$<\frac{\delta}{2}+1-y_1^*(x)+1-y^*(y)<\frac{\delta}{2}+2\alpha<\delta.$$
\end{proof}

A similar result can be stated replacing ``weakly open sets" with ``convex combination of slices" (and consequently, with ``slices").

\begin{lemma}\label{renormaconvex}

Let $X$ be a Banach space. The following assertions are equivalent:

\begin{enumerate}
\item $B_X$ contains convex combinations of slices of arbitrarily small diameter.
\item $B_X$ contains convex combinations of slices of $S_X$ of arbitrarily small diameter.
\item $B_X$ contains convex combinations of Lip-slices of arbitrarily small diameter.
\end{enumerate}

\end{lemma}

\begin{proof}
(2)$\Leftrightarrow$(3) is known by Theorem \ref{caratopolo}.

(1)$\Rightarrow$(2) is clear.

(2)$\Rightarrow$(1) is similar to (2)$\Rightarrow$(1) in Lemma \ref{renormadebil}, but we shall provide a proof for the sake of completeness.

Consider an arbitrary $\delta>0$ and $C:=\sum_{i=1}^n \lambda_i S(S_X,x_i^*,\alpha)$ a convex combination of slices whose diameter is less than $\frac{\delta}{2}$. We can assume, up considering smaller numbers, that $0<\alpha<\min\{\frac{\delta}{4},1\}$. We shall prove that $diam\left(\sum_{i=1}^n \lambda_i S(B_X,x_i^*,\alpha)\right)\leq\delta$. To this aim pick $\sum_{i=1}^n \lambda_i x_i,\sum_{i=1}^n \lambda_i y_i\in \sum_{i=1}^n \lambda_i S(B_X,x_i^*,\alpha)$. Now it is clear that 
$$\sum_{i=1}^n \lambda_i \frac{x_i}{\Vert x_i\Vert},\sum_{i=1}^n \lambda_i \frac{y_i}{\Vert y_i\Vert}\in \sum_{i=1}^n \lambda_i S(S_X,x_i^*,\alpha)\Rightarrow \left\Vert \sum_{i=1}^n \lambda_i \frac{x_i}{\Vert x_i\Vert}-\sum_{i=1}^n \lambda_i \frac{y_i}{\Vert y_i\Vert}\right\Vert<\frac{\delta}{2}. $$ 
Now
$$\left\Vert \sum_{i=1}^n \lambda_i (x_i-y_i)\right\Vert$$
$$\leq \left\Vert \sum_{i=1}^n \lambda_i \frac{x_i}{\Vert x_i\Vert}-\sum_{i=1}^n \lambda_i \frac{y_i}{\Vert y_i\Vert}\right\Vert+\sum_{i=1}^n \lambda_i \left\Vert x_i-\frac{x_i}{\Vert x_i\Vert}\right\Vert+\sum_{i=1}^n \lambda_i \left\Vert y_i-\frac{y_i}{\Vert y_i\Vert}\right\Vert$$
$$<\frac{\delta}{2}+\sum_{i=1}^n \lambda_i \alpha+\sum_{i=1}^n \lambda_i \alpha=\frac{\delta}{2}+2\alpha<\delta.$$
Consequently, $diam\left(\sum_{i=1}^n \lambda_i S(B_X,x_i^*,\alpha)\right)\leq\delta$, as desired.
\end{proof}

Now using above Lemmata and Theorem \ref{teorenorma} we get the desired characterization of RNP, CPCP and SR.

\begin{theorem}\label{caralipsiso}

Let $X$ be a Banach space.

\begin{enumerate}
\item $X$ has the RNP if, and only if, for every equivalent norm on $X$ it follows that its new unit sphere contains Lip-slices of arbitrarily small diameter.

\item $X$ has the CPCP if, and only if, for every equivalent norm on $X$ it follows that its new unit sphere contains  non-empty relatively open subsets, for the Lip topology, of arbitrarily small diameter.  

\item $X$ is SR if, and only if, for every equivalent norm on $X$ it follows that its new unit ball contains convex combinations of Lip-slices of arbitrarily small diameter.
\end{enumerate}

\end{theorem}

The above result shows that RNP, CPCP and SR are properties which depend on the natural uniformity of the Banach space, given by the metric and the linear structure. In fact, roughly speaking, we can say that every property in Banach spaces characterized in terms of the diameter of relevant subsets for the weak topology is determined by the natural uniformity in the space and so the property also can be characterized in terms of relevant subsets for the Lip topology. This is too the case, as we have seen, for the diameter two properties and Daugavet property, for example.

\end{document}